\documentclass[11pt]{article}
\usepackage[utf8]{inputenc} 
\usepackage[T1]{fontenc} 
\usepackage{amsmath}
\usepackage{amsthm}
\usepackage{amsfonts}
\usepackage{amssymb}
\usepackage{lmodern} 
\usepackage{tikz}
\usetikzlibrary{arrows}

\usepackage{pgf,tikz,pgfplots}
\usepackage{mathrsfs}

\definecolor{ffqqqq}{rgb}{1,0,0}
\definecolor{wrwrwr}{rgb}{0.3803921568627451,0.3803921568627451,0.3803921568627451}
\definecolor{ffqqqq}{rgb}{1,0,0}
\definecolor{wrwrwr}{rgb}{0.3803921568627451,0.3803921568627451,0.3803921568627451}
\definecolor{rvwvcq}{rgb}{0.08235294117647059,0.396078431372549,0.7529411764705882}
\definecolor{cqcqcq}{rgb}{0.7529411764705882,0.7529411764705882,0.7529411764705882}
\definecolor{ffffff}{rgb}{1,1,1}
\definecolor{ccqqqq}{rgb}{0.8,0,0}

\usepackage[a4paper]{geometry}
\usepackage{cite}
\usepackage[english]{babel}

\newtheorem{theorem}{Theorem}[section]
\newtheorem{lemma}[theorem]{Lemma}
\newtheorem{proposition}[theorem]{Proposition}

\newcommand{\N}{{\mathbb N}}

\newcommand{\R}{{\mathbb R}}

\newcommand{\Int}{\mbox{Int}}
\newcommand{\Vol}{\mbox{Vol}}

\begin{document}

\title{Algebraic intersection for translation surfaces in the stratum $\mathcal{H}(2)$ \\
Intersection algébrique dans la strate $\mathcal{H}(2)$}
\author{ Smaïl Cheboui, Arezki Kessi, Daniel Massart}
\date{\today}
\maketitle
\begin{abstract}
Nous étudions la quantité $\mbox{KVol}$ définie par l'équation  (\ref{defKVol}) sur la strate  $\mathcal{H}(2)$ des surfaces de  translation de genre  $2$, avec une singularité conique. Nous donnons une suite explicite de surfaces $L(n,n)$ telles que  $\mbox{KVol}(L(n,n)) \longrightarrow 2$ quand $n$ tend vers l'infini,  $2$ étant l'infimum-conjectural-de $\mbox{KVol}$ sur $\mathcal{H}(2)$.

We study the quantity $\mbox{KVol}$ defined in Equation (\ref{defKVol}) on the stratum $\mathcal{H}(2)$ of translation surfaces of genus $2$, with one conical point. We provide an explicit sequence $L(n,n)$ of surfaces such that $\mbox{KVol}(L(n,n)) \longrightarrow 2$ when $n$ goes to infinity, $2$ being the conjectured infimum for  $\mbox{KVol}$ over $\mathcal{H}(2)$.

\end{abstract}
\section{Introduction}
Let $X$  be a closed surface, that is, a compact, connected manifold of dimension 2, without boundary. Let us assume that $X$ is oriented. Then the algebraic intersection of closed curves in $X$ endows the first homology $H_1(X,\R)$ with an antisymmetric, non degenerate,  bilinear form, which we denote $\mbox{Int}(.,.)$. 

Now let us assume $X$ is endowed with a Riemannian metric $g$. We denote $\mbox{Vol}(X,g)$ the Riemannian volume of $X$ with respect to the metric $g$, and for any piecewise smooth closed curve $\alpha$ in $X$, we denote $l_g(\alpha)$ the length of $\alpha$ with respect to $g$. When there is no ambiguity we omit the reference to $g$. 

We are interested in the quantity
\begin{equation}\label{defKVol}
\mbox{KVol}(X,g) = \Vol (X,g) \sup_{\alpha,\beta} \frac{\Int (\alpha,\beta)}{l_g (\alpha) l_g (\beta)}
\end{equation}
where the supremum ranges over all piecewise smooth closed curves $\alpha$ and $\beta$ in $X$. The $\Vol (X,g)$ factor is there to make $\mbox{KVol}$ invariant to re-scaling of the metric $g$. See \cite{MM} as to why $\mbox{KVol}$ is finite. It is easy to make $\mbox{KVol}$ go to infinity, you just need to pinch a non-separating closed curve $\alpha$ to make its length go to zero. The interesting surfaces are those $(X,g)$ for which $\mbox{KVol}$ is small. 

When $X$ is the torus, we have $\mbox{KVol} (X,g) \geq 1$, with equality if and only if the metric $g$ is flat (see \cite{MM}). Furthermore, when $g$ is flat, the supremum in (\ref{defKVol}) is not attained, but for a negligible subset of the set of all flat metrics.   In \cite{MM} $\mbox{KVol}$ is studied as a function of $g$, on the moduli space of hyperbolic  (that is, the curvature of $g$ is $-1$) surfaces of fixed genus. It is proved that $\mbox{KVol}$ goes to infinity when $g$ degenerates by pinching a non-separating closed curve, while $\mbox{KVol}$ remains bounded when $g$ degenerates by pinching a separating closed curve.

This leaves open the question whether $\mbox{KVol}$ has a minimum over the moduli space of hyperbolic surfaces of genus $n$, for $n \geq 2$. It is conjectured in \cite{MM} that for almost every $(X,g)$ in the moduli space of hyperbolic surfaces of genus $n$, the supremum in (\ref{defKVol})  is  attained (that is, it is actually a maximum). 

In this paper we consider a different class of surfaces : translation surfaces of genus 2, with one conical point. The set (or stratum) of such surfaces is denoted $\mathcal{H}(2)$ (see \cite{HL}). 
By  \cite{Mc}, any surface $X$ in the stratum $\mathcal{H}(2)$ may be unfolded as shown in Figure \ref{unfolding}, with complex parameters $z_1, z_2, z_3, z_4$. The surface is obtained from the plane template by identifying parallel sides of equal length. 
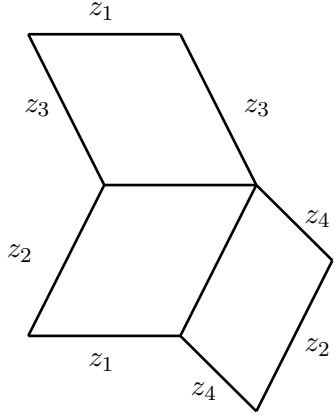
\begin{figure}
\begin{center}
\begin{tikzpicture}[scale=1]
\draw [line width=1pt] (0,0)-- (2,0);
\draw [line width=1pt]   (1,2)-- (3,2);
\draw [line width=1pt] (0,0)-- (1,2);
\draw [line width=1pt] (2,0)-- (3,2);
\draw [line width=1pt] (2,0)-- (3,-1);
\draw [line width=1pt] (3,-1)-- (4,1);
\draw [line width=1pt] (4,1)-- (3,2);
\draw [line width=1pt] (1,2)-- (0,4);
\draw [line width=1pt] (0,4)-- (2,4);
\draw [line width=1pt] (2,4)-- (3,2);
\draw (0.6567442569615567,-0.1) node[anchor=north west] {$z_1$};
\draw (0.6567442569615567,4.55122683020287476) node[anchor=north west] {$z_1$};
\draw (-0.418082138548259,1.344983575211053) node[anchor=north west] {$z_2$};
\draw (3.5,0.1) node[anchor=north west] {$z_2$};
\draw (-0.18540472752160073,3.276312254642753) node[anchor=north west] {$z_3$};
\draw (2.7,3.276312254642753) node[anchor=north west] {$z_3$};
\draw (2.003193663375158025,-0.4855801296416017) node[anchor=north west] {$z_4$};
\draw (3.5,1.8) node[anchor=north west] {$z_4$};
\end{tikzpicture}
\caption{Unfolding an element of $\mathcal{H}(2)$ }\label{unfolding}
\end{center}
\end{figure}

It is proved in  \cite{JP} (see also \cite{HMS}) that the systolic volume  has a minimum in $\mathcal{H}(2)$, and it is achieved by a translation surface tiled by six equilateral triangles. Since the systolic volume is a close relative of $\mbox{KVol}$, it is interesting to keep  the results of  \cite{JP} and  \cite{HMS} in mind. 

We have reasons to believe that $\mbox{KVol}$ behaves differently in $\mathcal{H}(2)$, both from the systolic volume in  $\mathcal{H}(2)$, and from  $\mbox{KVol}$ itself  in the moduli space of hyperbolic surfaces of genus $2$ ; that is, $\mbox{KVol}$ does not have a minimum over $\mathcal{H}(2)$. 

We also believe that the infimum of $\mbox{KVol}$ over  $\mathcal{H}(2)$  is $2$. This paper is a first step towards the proof : we find an explicit sequence $L(n,n)$ of surfaces in $\mathcal{H}(2)$, whose $\mbox{KVol}$ tends to $2$ (see Proposition \ref{Lnn}). These surfaces are obtained from very thin, symmetrical, L-shaped templates (see Figure \ref{fig_Lnn}). 

In the companion paper \cite{CKM} we study $\mbox{KVol}$ as a function on the Teichmüller disk (the $SL_2(\R)$-orbit) of surfaces in $\mathcal{H}(2)$ which are tiled by three identical parallelograms (for instance $L(2,2)$), and prove that $\mbox{KVol}$ does have a minimum there, but is not bounded from above. Therefore $\mbox{KVol}$  is not bounded from above as a function on $\mathcal{H}(2)$. In \cite{CKM} we also compute $\mbox{KVol}$ for the translation surface tiled by six equilateral triangles, and find it equals $3$, so it does not minimize $\mbox{KVol}$, neither in $\mathcal{H}(2)$, nor even in its own Teichmüller disk.

\section{$L(n,n)$}\label{sectionLnn}
\subsection{Preliminaries}
Following \cite{Schmithusen}, for any $n \in \N$, $n\geq 2$, we call $L(n+1,n+1)$  the $(2n+1)$-square translation surface of genus two, with one conical point,  depicted in Figure \ref{fig_Lnn}, where the upper and rightmost rectangles are made up with $n$ unit squares. We call $A$ (resp. $B$) the region in $L(n+1,n+1)$ obtained, after identifications,  from the uppermost (resp. rightmost) rectangle, and $C$ the region in $L(n+1,n+1)$ obtained, after identifications,  from the bottom left square. Both $A$ and $B$ are annuli with a pair of points identified on the boundary, while $C$ is a square with all four corners identified. We call  $e_1, e_2,$ (resp. $f_1, f_2$)  the closed curves in  $L(n+1,n+1)$  obtained by gluing the endpoints of the horizontal (resp. vertical) sides of $A$ and $B$. The closed curve which sits on the opposite side of $C$ from $e_1$ (resp. $f_1$) is called $e'_1$ (resp. $f'_1$), it is homotopic to $e_1$ (resp. $f_1$) in  $L(n+1,n+1)$. The closed curves in $L(n+1,n+1)$ which correspond to the diagonals of the square $C$ are called $g$ and $h$. 

Figure \ref{local pic} shows a local picture of $L(n+1,n+1)$ around the singular (conical) point $S$, with angles rescaled so the $6\pi$ fit into $2\pi$.

Since $e_1, e_2, f_1,  f_2$ do not meet anywhere but at $S$, the local picture yields the algebraic intersections between any two of $e_1, e_2, f_1,  f_2$, summed up in the following matrix: 
\begin{equation}\label{matrice_intersection}
\begin{array}{ccccc}
\mbox{Int} & e_2& f_1 & e_1 & f_2 \\
e_2            & 0     &   1   &  0   &  -1    \\
f_1            & -1    &   0   &  0   &  0  \\
e_1             & 0     &   0  & 0    &  1   \\
f_2             & 1    &  0    & -1   &  0 
\end{array}
\end{equation}

We call $T_A$ (resp. $T_B$) the flat torus obtained by gluing the opposite sides of the rectangle made with the $n+1$ leftmost squares (resp. with the $n+1$ bottom squares), so the homology of $T_A$ (resp. $T_B$)  is generated by $e_1$ and the concatenation of $f_1$ and $f_2$ (resp. $f_1$ and the concatenation of $e_1$ and $e_2$).

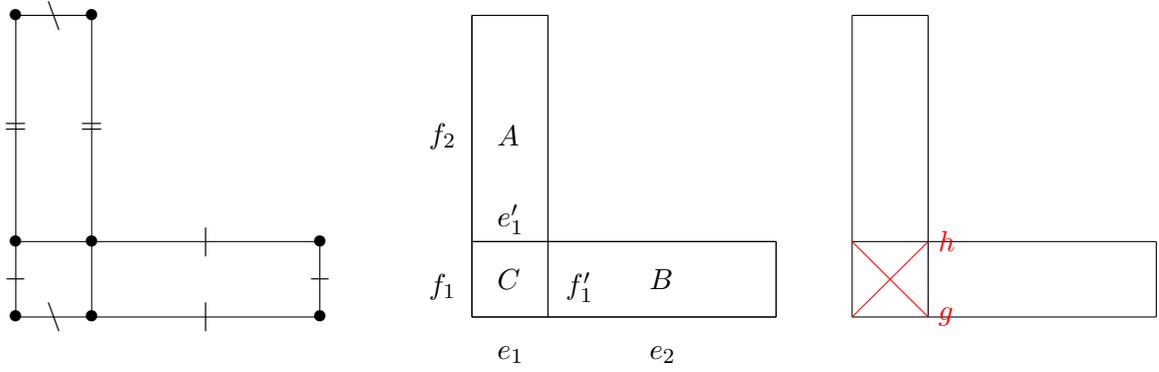
\begin{figure}
\begin{center}

\begin{tikzpicture}
\draw (0,0) -- (4,0);
\draw (0,0) -- (0,4);
\draw (1,0) -- (1,4);
\draw (0,1) -- (4,1);
\draw (4,0) -- (4,1);
\draw (0,4) -- (1,4);
\draw (0,0) node {$\bullet$} ;
\draw (1,0) node {$\bullet$} ;
\draw (4,0) node {$\bullet$} ;
\draw (0,1) node {$\bullet$} ;
\draw (1,1) node {$\bullet$} ;
\draw (4,1) node {$\bullet$} ;
\draw (0,4) node {$\bullet$} ;
\draw (1,4) node {$\bullet$} ;
\draw (0.5,0) node {$\backslash $};
\draw (0.5,4) node {$\backslash $};
\draw (2.5,0) node {$|$};
\draw (2.5,1) node {$|$};
\draw (0,0.5) node {$- $};
\draw (4,0.5) node {$-$};
\draw (0,2.5) node {$=$};
\draw (1,2.5) node {$=$};
\draw (6,0) -- (10,0);
\draw (6,0) -- (6,4);
\draw (7,0) -- (7,4);
\draw (6,1) -- (10,1);
\draw (10,0) -- (10,1);
\draw (6,4) -- (7,4);
\draw (6.2, -0.5) node[right] {$e_1$};
\draw (6.2, 1.3) node[right] {$e'_1$};
\draw (8.2, -0.5) node[right] {$e_2$};
\draw (5.3, 0.4) node[right] {$f_1$};
\draw (7.1, 0.4) node[right] {$f'_1$};
\draw (5.3, 2.4) node[right] {$f_2$};
\draw (6.2, 0.5) node[right] {$C$};
\draw (6.2, 2.4) node[right] {$A$};
\draw (8.2, 0.5) node[right] {$B$};

\draw (11,0) -- (15,0);
\draw (11,0) -- (11,4);
\draw (12,0) -- (12,4);
\draw (11,1) -- (15,1);
\draw (15,0) --  (15,1);
\draw (11,4) -- (12,4);
\draw (6,0) -- (10,0);
\draw (6,0) -- (6,4);
\draw (7,0) -- (7,4);
\draw (6,1) -- (10,1);
\draw (10,0) -- (10,1);
\draw[color=red] (11,0) -- (12,1) node[right, color=red]{$h$};
\draw[color=red] (11,1) -- (12,0) node[right, color=red]{$g$};

\end{tikzpicture}
\caption{$L(n+1,n+1)$}\label{fig_Lnn}
\end{center}
\end{figure}

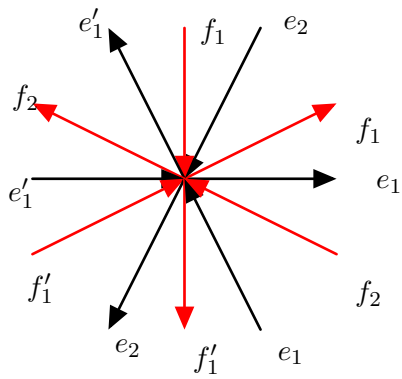
\begin{figure}
\begin{center}
\begin{tikzpicture}[line cap=round,line join=round,>=triangle 45,x=1cm,y=1cm]
\clip(-3.36621647446458,-3.1902) rectangle (3.786216474464581,2.4702);
\draw [->,line width=1pt,color=black] (0,0) -- (2,0);
\draw [->,line width=1pt,color=red] (0,0) -- (2,1);
\draw [->,line width=1pt,color=black] (1,2) -- (0,0);
\draw [->,line width=1pt,color=red] (0,2) -- (0,0);
\draw [->,line width=1pt,color=black] (0,0) -- (-1,2);
\draw [->,line width=1pt,color=red] (0,0) -- (-2,1);
\draw [->,line width=1pt,color=black] (-2,0) -- (0,0);
\draw [->,line width=1pt,color=red] (-2,-1) -- (0,0);
\draw [->,line width=1pt,color=black] (0,0) -- (-1,-2);
\draw [->,line width=1pt,color=red] (0,0) -- (0,-2);
\draw [->,line width=1pt,color=black] (1,-2) -- (0,0);
\draw [->,line width=1pt,color=red] (2,-1) -- (0,0);
\draw (2.3874355848434936,0.17619934102141693) node[anchor=north west] {$e_1$};
\draw (2.107679406919276,0.9408662273476115) node[anchor=north west] {$f_1$};
\draw (1.165833607907744,2.2836958813838555) node[anchor=north west] {$e_2$};
\draw (0.06545930807248848,2.2557202635914337) node[anchor=north west] {$f_1$};
\draw (-1.5291509060955513,2.3955983525535425) node[anchor=north west] {$e'_1$};
\draw (-2.4057202635914328,1.3791509060955522) node[anchor=north west] {$f_2$};
\draw (-2.452346293245469,0.13889851729818795) node[anchor=north west] {$e'_1$};
\draw (-2.2098909390444805,-1.1200042833607908) node[anchor=north west] {$f'_1$};
\draw (-1.0535654036243816,-2) node[anchor=north west] {$e_2$};
\draw (-0.027792751235584022,-2.061850082372323) node[anchor=north west] {$f'_1$};
\draw (1.0912319604612861,-2.099150906095552) node[anchor=north west] {$e_1$};
\draw (2.107679406919276,-1.2225815485996707) node[anchor=north west] {$f_2$};
\end{tikzpicture}
\caption{Local picture around the conical point}\label{local pic}
\end{center}
\end{figure}

\begin{lemma}\label{pas_inter_e1}
The only closed geodesics in $L(n+1,n+1)$ which do not intersect $e_1$ nor $f_1$ are, up to homotopy,  $e_1$, $f_1$, $g$, and $h$.
\end{lemma}
\begin{proof}
Let $\gamma$ be such a closed geodesic. It cannot enter, nor leave, $A$, $B$, nor $C$. If it is contained in $A$, and does not intersect $e_1$, then it must be homotopic to $e_1$, which is the soul of the annulus from which $A$ is obtained by identifying two points on the boundary. Likewise, if it is contained in $B$, and does not intersect $f_1$, then it must be homotopic to $f_1$. Finally, if $\gamma$ is not contained in $A$ nor in $B$, it must be contained in $C$. The only closed geodesics contained in $C$ are the sides and diagonals of the square from which $C$ is obtained, which are $e_1$, $e'_1$,  $f_1$, $f'_1$, $g$, and $h$.
\end{proof}

\begin{lemma}\label{inter_e1}
For any closed geodesic $\gamma$  in $L(n+1,n+1)$, we have  $l(\gamma) \geq n |\mbox{Int}(\gamma, e_1)|$.
\end{lemma}
\begin{proof}
For each intersection with $e_1$, $\gamma$ must go through $A$, from boundary to boundary.
\end{proof}
Obviously a similar lemma holds with $f_1$ instead of $e_1$. For $g$ and $h$ the proof is a bit different : 
\begin{lemma}\label{inter_g}
For any closed geodesic $\gamma$  in $L(n+1,n+1)$, we have  $l(\gamma) \geq n |\mbox{Int}(\gamma, g)|$.
\end{lemma}
\begin{proof}
First, observe that between two consecutive intersections with $g$, $\gamma$ must go through either $A$ or $B$, unless $\gamma$ is $g$ itself, or $h$ : indeed, the only geodesic segments contained in $C$ with endpoints on $g$ are segments of $g$, or $h$.  Obviously $\mbox{Int}(g,g)=0$, and from the intersection matrix (\ref{matrice_intersection}), knowing that $\left[g\right] = \left[ e_1 \right]-\left[  f_1 \right]$, $\left[ h \right] = \left[ e_1 \right] + \left[  f_1 \right]$, we see that $\mbox{Int}(g,h)=0$.

Thus, either $\mbox{Int}(\gamma, g)=0$, or each intersection must be paid for with a trek through $A$ or $B$, of length at least $n$.
\end{proof}
Obviously a similar lemma holds with $h$ instead of $g$. Note that Lemmata \ref{pas_inter_e1}, \ref{inter_e1}, \ref{inter_g} imply that the only geodesics in $L(n+1,n+1)$ which are shorter than $n$ are $e_1$, $f_1$, $g$, $h$, and closed geodesics homotopic to $e_1$ or $f_1$.

\begin{lemma}\label{sommes}
Let $I,J$ be positive integers, take $a_{ij}, i=1,\ldots , I, j=1,\ldots, J$ in $\R_+$, and $b_1, \ldots ,b_I, c_1, \ldots ,c_J$ in $\R_+^*$. Then we have
\[
\frac{  \sum_{i,j} a_{ij}    }{ \left(\sum_{i =1}^I b_i \right)  \left(\sum_{j =1}^J c_j \right) } \leq \max_{i,j} \frac{a_{ij}}{b_i c_j}.
\]
\end{lemma}
\begin{proof}
Re-ordering, if needed, the $a_{ij}, b_i, c_j$, we may assume 
\[
 \frac{a_{ij}}{b_i c_j} \leq \frac{a_{11}}{b_1 c_1} \  \forall i=1,\ldots, I, j=1,\ldots, J.
 \]
  Then $a_{ij}b_1 c_1 \leq a_{11} b_i c_j $ $\forall i=1,\ldots ,I, j=1,\ldots, J$, so 
\[
b_1 c_1 \sum_{i,j}a_{ij} \leq a_{11} \sum_{i,j} b_i c_j =a_{11} \left(\sum_{i =1}^I b_i \right)  \left(\sum_{j =1}^J c_j \right).
\]
\end{proof}

\subsection{Estimation of $\mbox{KVol}(L(n,n))$}
\begin{proposition}\label{Lnn}
\[
\lim_{n \rightarrow +\infty}  \mbox{KVol}(L(n+1,n+1)) = 2.
  \]
\end{proposition}
\begin{proof}
First observe that $ \mbox{Vol}(L(n+1,n+1)) = 2n+1$, $l(e_1)=1$, $l(f_2)=n$, $\mbox{Int}(e_1, f_2)=1$, so 
\[
  \mbox{KVol}(L(n+1,n+1)) \geq 2+\frac{1}{n}.
  \]
To bound $ \mbox{KVol}(L(n+1,n+1))$ from above, we take two closed geodesics $\alpha$ and $\beta$ ; by Lemmata \ref{inter_e1}, \ref{inter_g}, if either $\alpha$ or $\beta$ is homotopic to $e_1$, $f_1$, $g$, or $h$, then 
\[
\frac{\mbox{Int}(\alpha,\beta)}{l(\alpha)l(\beta)} \leq \frac{1}{n},
\]
so from now on we assume that neither $\alpha$ or $\beta$ is homotopic to $e_1$, $f_1$, $g$, $h$. We cut $\alpha$ and $\beta$ into pieces using the following procedure : we consider the sequence of intersections of $\alpha$ with $e_1, e'_1, f_1, f'_1$, in cyclical order, and we cut $\alpha$ at each intersection with $e_1$ or  $e'_1$ which is followed by an intersection with $ f_1$ or $ f'_1$, and at each intersection with $f_1$ or  $f'_1$ which is followed by an intersection with $ e_1$ or $ e'_1$. We proceed likewise with $\beta$. We call $\alpha_i, i =1, \ldots, I$, and $\beta_j, j =1, \ldots ,J$, the pieces of $\alpha$ and $\beta$, respectively. 

Note that 
\[
l(\alpha) = \sum_{i=1}^{I} l(\alpha_i), \   l(\beta) = \sum_{j=1}^{J} l(\beta_j), \mbox{ and }
\]
\[ |\mbox{Int}(\alpha,\beta)| \leq  \sum_{i,j} |\mbox{Int}(\alpha_i,\beta_j)|,
\]
so Lemma \ref{sommes} says that 
\[
\frac{|\mbox{Int}(\alpha,\beta)|}{l(\alpha) l(\beta) } \leq \max_{i,j} \frac{ |\mbox{Int}(\alpha_i,\beta_j)|}{l(\alpha_i) l(\beta_j) }.
\]

We view each piece $\alpha_i$ (resp. $\beta_j$) as a geodesic arc in the torus $T_A$ (resp.  $T_B$), with endpoints on the image in $T_A$ (or $T_B$) of $f_1$ or $f'_1$ (resp. $e_1$ or $e'_1$), which is a geodesic arc of length $1$, so we can close each $\alpha_i$ (resp. $\beta_j$) with a piece of $f_1$ or $f'_1$ (resp. $e_1$ or $e'_1$), of length $\leq 1$. We choose a closed geodesic  $\hat{\alpha}_i$ (resp. $\hat{\beta}_j$) in $T_A$ (resp. $T_B$) which is homotopic  to the closed curve thus obtained. We have $l(\hat{\alpha}_i) \leq l(\alpha_i)+1$, $l(\hat{\beta}_j) \leq l(\beta_j)+1$, so 
\[
\frac{1}{l(\hat{\alpha}_i)l(\hat{\beta}_j) } \geq \frac{1}{(l(\alpha_i)+1)(l(\beta_j)+1)}.
\]
Now recall that $l(\alpha_i), l(\beta_j) \geq n$,  so $l(\alpha_i)+1 \leq (1+\frac{1}{n}) l(\alpha_i)$, whence
\[
\frac{1}{l(\hat{\alpha}_i)l(\hat{\beta}_j) } \geq \frac{1}{l(\alpha_i)l(\beta_j)} \left( \frac{n}{n+1} \right)^2.
\]
Next, observe that $ |\mbox{Int}(\alpha_i,\beta_j)| \leq  |\mbox{Int}(\hat{\alpha_i},\hat{\beta_j})|+1$, because $\hat{\alpha}_i$ (resp.   $\hat{\beta}_j$) is homologous to a closed curve which contains $\alpha_i$  (resp. $\beta_j$) as a subarc, and the extra arcs cause at most one extra intersection, depending on whether or not the endpoints of $\alpha_i$ and $\beta_j$ are intertwined. So,
\[
 \frac{ |\mbox{Int}(\alpha_i,\beta_j)|}{l(\alpha_i) l(\beta_j) } \leq \frac{ |\mbox{Int}(\hat{\alpha_i},\hat{\beta_j})|+1}{l(\hat{\alpha}_i)l(\hat{\beta}_j) }  \left( \frac{n+1}{n} \right)^2
 \leq  \left( \frac{ |\mbox{Int}(\hat{\alpha_i},\hat{\beta_j})|}{l(\hat{\alpha}_i)l(\hat{\beta}_j) }+ \frac{1}{n^2} \right)  \left( \frac{n+1}{n} \right)^2   ,  
 \]
where the last inequality stands because $l(\hat{\alpha}_i)\geq n$, $ l(\hat{\beta}_j) \geq n$, since $\hat{\alpha}_i$ and $\hat{\beta}_j$ both have to go through a cylinder $A$ or $B$ at least once.
Finally, since $\hat{\alpha}_i$ and  $\hat{\beta}_j$ are closed geodesics on a flat torus of volume $n+1$, we have (see \cite{MM})
\[
 \frac{ |\mbox{Int}(\hat{\alpha_i},\hat{\beta_j})|}{l(\hat{\alpha}_i)l(\hat{\beta}_j) } \leq \frac{1}{n+1}, \mbox{ so }
 \]
 \[
  \frac{ |\mbox{Int}(\alpha_i,\beta_j)|}{l(\alpha_i) l(\beta_j) } \leq  \left(  \frac{1}{n+1}+ \frac{1}{n^2} \right)  \left( \frac{n+1}{n} \right)^2 = \frac{1}{n} + o\left(\frac{1}{n}\right)  , 
  \]
which yields the result, recalling that $ \mbox{Vol}(L(n+1,n+1)) = 2n+1$.
\end{proof}

\noindent \textbf{adresses} : 

\noindent Smaïl Cheboui :  USTHB, Facult\'e de Math\'ematiques, Laboratoire de Syst\`emes Dynamiques, 16111 El-Alia BabEzzouar - Alger, Alg\'erie, 

\noindent email : smailsdgmath@gmail.com

\noindent Arezki Kessi :       USTHB, Facult\'e de Math\'ematiques, Laboratoire de Syst\`emes Dynamiques, 16111 El-Alia BabEzzouar - Alger, Alg\'erie, 

\noindent email: arkessi@yahoo.fr

\noindent Daniel Massart : Institut Montpelli\'erain Alexander Grothendick, CNRS Universit\'e de  Montpellier,  France 

\noindent email : daniel.massart@umontpellier.fr (corresponding author)

\end{document}